\documentclass[12pt]{amsart}

\usepackage[all]{xy}                        %
\CompileMatrices                            
\UseTips                                    

\input xypic
\usepackage[bookmarks=true]{hyperref}       

\usepackage{amssymb,latexsym,amsmath,amscd}
\usepackage{xspace}
\usepackage{enumerate}
\usepackage{graphicx}
\usepackage{vmargin}
\usepackage{todonotes}

\usepackage{dcpic,pictexwd}
\usepackage{tcolorbox}
\usepackage{subcaption}

\DeclareMathOperator{\Res}{Res}

\theoremstyle{plain}
\newtheorem{theorem}{Theorem}[section]
\newtheorem*{theorem*}{Theorem}
\newtheorem{proposition}[theorem]{Proposition}
\newtheorem{corollary}[theorem]{Corollary}
\newtheorem{lemma}[theorem]{Lemma}

\theoremstyle{definition}
\newtheorem{definition}[theorem]{Definition}

\newtheorem{remark}[theorem]{Remark}

\newcommand{\enm}[1]{\ensuremath{#1}}          %

\newcommand{\cal}[1]{\mathcal{#1}}

\renewcommand{\bar}[1]{\overline{#1}}

\newcommand{\CC}{\enm{\mathbb{C}}}

\newcommand{\NN}{\enm{\mathbb{N}}}
\newcommand{\RR}{\enm{\mathbb{R}}}

\newcommand{\ZZ}{\enm{\mathbb{Z}}}
\newcommand{\FF}{{\enm{\mathbb{F}}}}

\newcommand{\Cc}{\enm{\cal{C}}}

\newcommand{\Ii}{\enm{\cal{I}}}

\newcommand{\Ll}{\enm{\cal{L}}}

\newcommand{\Oo}{\enm{\cal{O}}}
\newcommand{\Pp}{\enm{\cal{P}}}

\newcommand{\Tt}{\enm{\cal{T}}}

\DeclareMathOperator{\bdeg}{bdeg}

\renewcommand{\phi}{\varphi}
\renewcommand{\theta}{\vartheta}
\renewcommand{\epsilon}{\varepsilon}

\def\8{\infty}

\def\ge{\geqslant}
\def\le{\leqslant}

\def\PP{\mathbb{P}}

\def\PPPP{{\PP^{2}\times\PP^{2}}}

\def\Res{\mathop{\mathrm{Res}}}
\def\Pic{\mathop{\mathrm{Pic}}}

\usepackage{color}

\let\oldtocsubsection=\tocsubsection
\renewcommand{\tocsubsection}[2]{\hskip15pt\oldtocsubsection{#1}{#2}}

\begin{document}

\title[Surfaces in $\mathbb{F}$ containing
smooth conics and twistor fibers]{Surfaces in the flag threefold  containing\\
smooth conics and twistor fibers
}

\author[A. Altavilla]{Amedeo Altavilla}\address{Dipartimento di Matematica,
  Universit\`a degli Studi di Bari `Aldo Moro', via Edoardo Orabona, 4, 70125,
  Bari, Italia}\email{amedeo.altavilla@uniba.it}

\author[E. Ballico]{Edoardo Ballico}\address{Dipartimento Di Matematica,
  Universit\`a di Trento, Via Sommarive 14, 38123, Povo, Trento, Italia}
\email{edoardo.ballico@unitn.it}

\author[M. C. Brambilla]{Maria Chiara Brambilla}\address{
Universit\`a Politecnica delle Marche, via Brecce Bianche, I-60131 Ancona, Italia}
\email{brambilla@dipmat.univpm.it}

\thanks{Partially supported by GNSAGA of INdAM and by the INdAM project `Teoria
  delle funzioni ipercomplesse e applicazioni'}



\begin{abstract} 
We study smooth integral curves of bidegree $(1,1)$, called \textit{smooth conics}, in the flag threefold $\mathbb{F}$.
The study is motivated by the fact that the family of smooth conics
contains the set of fibers of the twistor projection $\mathbb{F}\to\mathbb{CP}^{2}$.
We give a bound on the maximum number of smooth conics contained in
a smooth surface $S\subset\mathbb{F}$. Then, we show qualitative properties of algebraic surfaces containing a prescribed number of
smooth conics.
Lastly, we study surfaces containing infinitely many twistor fibers. We show
that the only smooth cases are surfaces of bidegree $(1,1)$. Then, for any
integer $a>1$, we exhibit a method to construct an integral surface
of bidegree $(a,a)$ containing infinitely many twistor fibers.
\end{abstract}

\maketitle
\setcounter{tocdepth}{1} 

\tableofcontents

\section{Introduction}

Let $\FF\subset \PP^2\times\PP^{2\vee}$ be the complex flag variety of points and lines contained in a plane. 
In this paper we investigate integral {(i.e.\ irreducible and reduced)} smooth curves of bidegree $(1,1)$, called \textit{smooth conics}, contained in a surface $S\subset \FF$.
The interest in this subject emerges from the fact that this family of
curves contains the set of fibers of the standard twistor projection
$\pi:\FF\to\PP^{2}$ \cite{AHS, Hitchin, ABBS}.
In fact, from twistor theory, the fibers of $\pi$ can be characterized to be the $j$-invariant rational smooth curves with normal bundle $\Oo(1)\oplus\Oo(1)$, 
where $j:\FF\to\FF$ is an anti-holomorphic involution involved in the twistor construction (see Formula~\eqref{mapj}).

Continuing from~\cite{ABBS} and inspired from~\cite{altavillaballico1, altavillaballico3}, we aim to obtain results on the maximum number of smooth
conics that can be contained in a smooth surface, on the geometry of the
generic surface containing a certain amount of smooth conics and a 
description of the surfaces containing infinitely many twistor fibers.

In the preliminary section, we collect 
some known results on $\FF$ and its subvarieties. 
In particular,
we define the bidegree of a curve and of a surface in $\FF$, we define
 smooth conics and twistor fibers and we present a detailed study of
the linear sections, the $(1,0)$ and $(0,1)$ surfaces, which are
Hirzebruch surfaces of type $1$. For other relevant geometry of the flag we refer
to~\cite{ABBS}.

In Section~\ref{bound}, we give a bound on the number of disjoint smooth conics contained in a surface $S$ of bidegree $(a,b)$ with both $a$ and $b$ greater or equal than $3$. Explicitly,  {the bound computed in  Theorem \ref{s1} is:}
$$
\frac{2 (a + b - 2) (3 a^2 b - a^2 + 3 a b^2 - 4 a b + 3 a - b^2 + 3 b)}{(a + b - 1)^2}.
$$ 
Since twistor fibers are disjoint smooth conics, the same bound also holds for them. The proof is 
based on a classical result of Miyaoka~\cite{miyaoka} and it is obtained
by means of formula involving Chern classes of a surface.
The same proof can also be adapted in order to compute
the maximum number of curves of bidegree $(1,0)$ or $(0,1)$ contained in a smooth surface (see Remark~\ref{3.7}).

We expect that if $m$ is the maximal number of disjoint smooth conics contained in a smooth surface $S\subset \FF$, then not all the integers $0\le k\le m$ 
{are realized by some smooth surface of bidegree $(a,b)$, at least for high $a$ and $b$. So an interesting open question is to find the maximum integer $m'\le m$ such that for all $0\le k\le m'$, there exists a smooth surface which contains $k$ disjoint smooth conics.
}

In Section~\ref{prescribed}, as a second main result, we give a range for an integer $x$  such that
there exists an irreducible smooth surface in the flag which contains precisely $x$ disjoint smooth conics.
More precisely we have the following theorem, in which we denote by $|\Ii_{T,\FF}(a,b)|$ the linear system of the surfaces of bidegree $(a,b)$ contained in $\FF$ that contains a set $T$ of disjoint conics.

\begin{theorem}\label{c1-seconda parte}
Fix integers $b\ge a\ge2$ and $0\le x\le (a-1)(a-2)/2$, and let $T\subset \FF$ be a general union of $x$ smooth conics.
Then:
\begin{enumerate}
\item[(i)] a general $S\in |\Ii_{T,\FF}(a,b)|$ does not contain any smooth conic $C$ such that $C\cap T=\emptyset$;
\item[(ii)] a general $S\in |\Ii_{T,\FF}(a,b)|$ is irreducible and smooth.
\end{enumerate}
\end{theorem}

The proof of such a result is obtained by means of the classical {\it Horace method}, \cite{HH}.
Moreover in the same section we prove a unicity result
for surfaces in $\FF$ which contain a given amount of smooth conics,
see Proposition \ref{a0}.

Finally, in Section~\ref{infinite-sec}, we study surfaces in $\FF$ containing infinitely many twistor fibers.
More precisely, in Proposition \ref{f4} we prove that if the singular locus of the surface is finite, then a surface containing infinitely many twistor fibers must be a smooth surface of bidegree $(1,1)$. On the other hand,
in Theorem \ref{infinite} we provide a construction of singular surfaces of bidegree $(a,a)$ for any $a\ge2$ containing infinitely many twistor lines. In particular, for any $n\le a-1$, we give examples of surfaces $S$ of bidegree $(a,a)$ for which the family of twistor lines 
contained in $S$ is parametrized by 
$n$ disjoint circles,
up to finitely many identifications and finitely many isolated points.

Collecting the results of the final section, we get that, for any positive integer $a$ there exists a surface of bidegree $(a,a)$ containing infinitely many twistor lines, belonging to a continuous family. 
These results resemble the behavior of the other basic example of algebraic twistor space, which is $\PP^{3}$ projecting onto $S^{4}$. In that case, considering the results in~\cite{sv, altavillaballico3}, we obtained that the only smooth case of surface containing infinitely many twistor fibers is that of quadrics, while the other admissible cases are only ruled surfaces of even degree. Here, in a sense, the analogy is complete, as the ``total degree'' in the Segre embedding of
a surface of bidegree $(a,b)$ is $a+b$.

\section{Preliminaries}\label{preliminaries}
In the whole paper we will take advantage of notations and basic facts already introduced in~\cite{ABBS}. Nevertheless, we
recall here briefly what is needed in order to be as more self-contained as possible. 

Let us denote by $\PP^{n}=\PP(\CC^{n+1})$ the complex projective space. {Given a complex projective variety $V$, we write $V_\RR$ for its real locus.} 
In the paper we will identify $\PP^2$ and its dual $\PP^{2\vee}$, both on $\RR$ and on $\CC$. This is possible because, if $P$ is a real algebraic
variety whose complexification is isomorphic to $\PP^{2}$, then $P$ is isomorphic to the real projective plane $\PP_\RR^2$ by \cite[Prop. VI (1.1) p. 109]{silhol}.

We will need in the \textit{Segre variety}  $\Sigma=\sigma(\PP^{2}\times \PP^{2})\subset \PP^{8}$, which 
is the image of the Segre embedding $\sigma:\PP^{2}\times \PP^{2}\to\PP^{8}$, defined as follows:
$$
(p,\ell)=([p_{0}:p_{1}:p_{2}],[\ell_{0}:\ell_{1}:\ell_{2}])
\mapsto \left[\begin{matrix} \ell_{0}\\ \ell_{1}\\ \ell_{2}
\end{matrix}
\right]\cdot[p_{0}:p_{1}:p_{2}]=\left[\begin{matrix} p_{0}\ell_{0} & p_{1}\ell_{0} & p_{2}\ell_{0}\\
p_{0}\ell_{1} & p_{1}\ell_{1} & p_{2}\ell_{1}\\
p_{0}\ell_{2} & p_{1}\ell_{2} & p_{2}\ell_{2}
\end{matrix}
\right].
$$
Recall that $\Sigma$ is
a $4$-fold of degree $6$ contained in $\PP^{8}$.

We denote the standard projections by $\Pi_{1},\Pi_{2}:\PPPP\to\PP^{2}$, where $\Pi_{1}(p,\ell)=p$ and $\Pi_{2}(p,\ell)=\ell$. 
For any $d_{1},d_{2}\in\ZZ$, we set $\Oo_{\PPPP}(d_{1},d_{2})=\Pi_{1}^{*}\Oo_{\PP^{2}}(d_{1})\otimes\Pi_{2}^{*}\Oo_{\PP^{2}}(d_{2})$. Recall that 
$\Pic(\PPPP)\simeq\ZZ^{2}$ and it is freely generated as an abelian group by
$\Oo_{\PPPP}(1,0)$ and $\Oo_{\PPPP}(0,1)$. 
The canonical line bundle is
$
\omega_{\PPPP}\simeq\Oo(-3,-3).
$

%

The {flag threefold} in $\PP^{2}\times \PP^{2\vee}$ is defined geometrically as the set of couples $(p,\ell)$ such that $p\in\ell$. 
However, we will mainly use the following algebraic definition as {incidence variety} in $\PP^2\times \PP^2$. 
Let us denote the elements of the first $\PP^{2}$ as rows, while those of the second one as columns, i.e.: if $(p,\ell)\in\PPPP$, then
$$
p=[p_{0}:p_{1}:p_{2}],\quad \ell=[\ell_{0}:\ell_{1}:\ell_{2}]^{\top}=\left[\begin{matrix}\ell_{0}\\ \ell_{1}\\ \ell_{2}
\end{matrix}
\right],
$$
so that $p\ell=p_{0}\ell_{0}+p_{1}\ell_{1}+p_{2}\ell_{2}$.
\begin{definition}
The \textit{flag threefold} is the algebraic subvariety of $\PPPP$ defined by
$$
\FF:=\{(p,\ell)=([p_{0}:p_{1}:p_{2}],[\ell_{0}:\ell_{1}:\ell_{2}]^{\top})\in\PPPP\,|\,p\ell=p_{0}\ell_{0}+p_{1}\ell_{1}+p_{2}\ell_{2}=0\}.
$$
\end{definition}
We define the standard projections $\pi_{1},\pi_{2}:\FF\to\PP^{2}$ as $\pi_{1}:=\Pi_{1|\FF}$ and $\pi_{2}:=\Pi_{2|\FF}$, and we set
$
\Oo_{\FF}(d_{1},d_{2}):=\pi_{1}^{*}\Oo(d_{1})\otimes\pi_{2}^{*}\Oo(d_{2}).
$

Since $\FF$ is an effective divisor belonging to $|\Oo_{\PPPP}(1,1)|$, by the adjunction formula we have 
\begin{equation}\label{omegaF}
\omega_{\FF}\simeq\Oo_{\FF}(-2,-2).
\end{equation}

%

We now  describe some geometry of curves and surfaces in $\FF$.
Given a line bundle $\mathcal L$ on a projective variety $V$ we denote by $|\mathcal L|=\PP(H^0(V,\mathcal L))$.

Let us start by defining the bidegree of a surface and of a curve in $\FF$.
\begin{definition}
Let $S\subset\FF$ be an algebraic surface in $|\Oo_{\FF}(d_{1},d_{2})|$. Then we say that $S$ has \textit{bidegree} $(d_{1},d_{2})$, and write
$\bdeg(S)=(d_{1},d_{2})$.
\end{definition}

We recall the
multiplication rules in the Chow ring: by~\cite[Proposition 3.11]{ABBS}, we have that
\begin{align*}
\Oo_{\FF}(1,0)\cdot\Oo_{\FF}(1,0)\cdot\Oo_{\FF}(1,0)=0,\\
\Oo_{\FF}(1,0)\cdot\Oo_{\FF}(0,1)\cdot\Oo_{\FF}(1,0)=1,\\
\Oo_{\FF}(0,1)\cdot\Oo_{\FF}(1,0)\cdot\Oo_{\FF}(0,1)=1,\\
\Oo_{\FF}(0,1)\cdot\Oo_{\FF}(0,1)\cdot\Oo_{\FF}(0,1)=0.
\end{align*}

While for surfaces the notion of bidegree is quite natural, in the case of curves we need a little bit more of elaboration.

\begin{definition}
Let $C\subset \FF$ be an algebraic curve. We define its {\it bidegree} as
$\bdeg(C)=(d_1,d_2)$, where if $\pi_i(C)=\{x\}$, then $d_{i}=0$, otherwise
$d_{i}=a_{i} b_{i}$ with $a_i=\deg(\pi_i(C))$ and $b_i=\deg(\pi_{i} |_{C})$.
\end{definition}

From the definition we have that $\pi_{1}^{-1}(p)$ has bidegree $(0,1)$, while
$\bdeg(\pi_{2}^{-1}(\ell))=(1,0)$. Moreover, all curves of bidegree $(0,1)$ (or $(1,0)$) can be represented as $\pi_{1}^{-1}(p)$ for some $p\in\PP^{2}$ (respectively, as $\pi_{2}^{-1}(\ell)$ for some $\ell\in\PP^{2}$).
 This notion of bidegree for curves is related to that of surfaces by means of general intersections. 

\begin{remark}\label{curved1d2}
Any curve $C$ of bidegree $(0,1)$ is obtained as the intersection of 
two surfaces of bidegree $(1,0)$ (see~\cite[Section 3.2]{ABBS}), i.e. $C$ is a class of type $\Oo_{\FF}(1,0)\cdot \Oo_{\FF}(1,0)$ in the Chow ring. Therefore, if $C$ has bidegree $(d_{1},d_{2})$, then it belongs to the class $d_{2}\Oo_{\FF}(1,0)\cdot \Oo_{\FF}(1,0)+d_{1}\Oo_{\FF}(0,1)\cdot \Oo_{\FF}(0,1)$.
Moreover, thanks to~\cite[Proposition 3.11]{ABBS}, we have that if $C$ has bidegree $(d_{1},d_{2})$, if and only if
$C\cdot\Oo_{\FF}(1,0)=d_{1}$ and $C\cdot\Oo_{\FF}(0,1)=d_{2}$.
\end{remark}

It was shown in~\cite[Remark 3.2]{ABBS} that if a curve $C$ has bidegree $(d_{1},d_{2})$ with $d_{i}=1$ (for some $i=1,2$), then $C$ is rational. A particularly interesting case is that of curves of bidegree $(1,1)$ as this family contains the set of twistor fibers. All these curves can be described as
$$
L_{q,m}:=\{(p,\ell)\in\FF\,|\,pm=0, q\ell=0\}.
$$
If $qm\neq0$, then $L_{q,m}$ is smooth and irreducible, while if $qm=0$, then
$L_{qm}=\pi_{1}^{-1}(q)\cup\pi_{2}^{-1}(m)$. 
Any curve $L_{q,m}$ has degree $2$ in the Segre embedding as it can be 
seen as a linear section of a smooth quadric.
We refer the reader to~\cite[Section 3.1]{ABBS} for a detailed analysis of these curves.
For the convenience of what follows we will set the following definition.

\begin{definition}
Let $L_{q,m}\subset\FF$ be a bidegree $(1,1)$ curve. If $qm\neq0$, then $L_{q,m}$ is said to be a \textit{smooth conic}.
The set of smooth conics will be denoted by $\Cc$.
\end{definition}

\begin{remark}\label{c5}
Fix a curve $F$ of bidegree $(1,0)$ and a curve $G$ of bidegree $(0,1)$ such that $F\cap G\ne \emptyset$. Note that $F\cap G$ is a unique point and that $F\cup G$ is a flat limit of a family of smooth conics, \cite[Section 2.1]{ABBS}. 
\end{remark}

From the definition of $L_{q,m}$, it is clear that the family of bidegree $(1,1)$
curves is parametrized by $\PP^{2}\times\PP^{2}$ and $\Cc$ is parametrized by $(\PP^{2}\times \PP^{2})\setminus \FF$.

As said before, $\Cc$ contains the set of twistor fibers $\pi^{-1}(q)$, where the twistor map 
$$
\pi: \FF\xrightarrow{}\PP^{2}
$$
is defined as $\pi(p,\ell)=\bar p\times \ell$ (see~\cite[Section 5]{ABBS} for more details).
If we denote by $j:\FF\to\FF$ the antiholomorphic involution map 
\begin{equation}\label{mapj}
j(p,\ell)=(\bar\ell,\bar p),
\end{equation}
then a twistor fiber is an integral  $(1,1)$ curve that is $j$-invariant.
In terms of $L_{q,m}$, we have that $j(L_{q,m})=L_{q,m}$ if and only if $m=\bar q$.
\begin{definition}
The set of twistor fibers will be denoted by $\Tt$.
\end{definition}

As said before $\Tt\subset\Cc$, but even more $\Tt$ is a Zariski dense in $
\Cc$ (see~\cite[Lemma 5.6]{ABBS}). 
\medskip

We now turn our attention to surfaces in $\FF$.
A first study of surfaces of bidegree $(1,0)$ and $(0,1)$ was carried out in~\cite[Section 3.2]{ABBS}. We start with the following definition.
\begin{definition}
Let $q,m\in\PP^{2}$. We define the following  surfaces in $\FF$:
\begin{align*}
H_{m}&:=\{(p,\ell)\in\FF\,|\,pm=0\}\in|\Oo_{\FF}(1,0)|,\\
_{q}H&:=\{(p,\ell)\in\FF\,|\,q\ell=0\}\in|\Oo_{\FF}(0,1)|.
\end{align*}
\end{definition}

We will not really need, in the current analysis, the nature of the $(1,0)$ or $(0,1)$ surfaces as $H_{m}$ or $ _{q}H$. In this view, we will denote throughout the paper
an element of $|\Oo_{\FF}(1,0)|$ as $X$ and an element of $|\Oo_{\FF}(0,1)|$ as $Y$.

\begin{remark}\label{intersezioni}
Recall that two general surfaces $X,X'\in|\Oo_{\FF}(1,0)|$ intersect in a fiber of $\pi_{1}$, while a general surface $X$ of bidegree $(1,0)$ and a general surface $Y$ of bidegree $(0,1)$ intersect in a smooth conic. Hence a general surface $X$ of bidegree $(1,0)$ (or $(0,1)$) intersects a general  conic in a general point, by \cite[Proposition 3.11]{ABBS}. 
\end{remark}

In~\cite[Section 3.2]{ABBS} we gave several geometric descriptions of 
these families of surfaces, noticing that any surface of bidegree $(1,0)$ is
a surface of type $H_{m}$ and any of bidegree $(0,1)$ is of type $ _{q}H$.
Moreover, as $H_{m}=\pi_{1}^{-1}(\pi_{1}(\pi_{2}^{-1}(m)))$, it is easy to see that $\pi_{2|H_{m}}:H_{m}\to\PP^{2}$ represents the blow-up of $\PP^{2}$ at $m$ (and analogously for $ _{q}H$ at $q$).
In particular, in~\cite[Corollary 3.10]{ABBS}, we proved that these are Hirzebruch surfaces of type $1$, usually denoted by $F_{1}$. 
For the purpose of this paper we will need to 
deepen their description as Hirzebruch surfaces.

Recall that $\Pic(F_1)=\ZZ h\oplus\ZZ f$, where 
$f$ is the fiber of the ruling of $F_1$ and $h$ is the section with self-intersection $-1$. 
and that $h^2=-1$, $f^2=0$ and $hf=1$. 
 
Now, given a surface $X$ of bidegree $(1,0)$, 
up to  identification of $X$ to  $F_1$, we get
that the line bundle $\Oo_X(1,0)$ corresponds to
$\Oo_{F_1}(f)$ (and to curves contained in $X$ of bidegree (0,1) in $\FF$) and the line bundle $\Oo_X(0,1)$
corresponds to $\Oo_{F_1}(h+f)$ (and to elements of $\Cc$). Hence
we get, for any $a,b\in\ZZ$ and  for any $\alpha,\beta\in\ZZ$:
\begin{equation}\label{coeff}
\Oo _X(a,b) \cong \Oo_{F_1}(bh+(a+b)f),\quad \mbox{ and }\quad  \Oo_{F_1}(\alpha h+\beta f)\cong \Oo _X(\beta-\alpha ,\alpha).\end{equation}

Clearly if $Y$ is a surface of bidegree $(0,1)$, identified again with a Hirzebruch surface $F_1$, we have analogously that
 $$\Oo _Y(a,b) \cong \Oo_{F_1}(ah+(a+b)f),\quad \mbox{ and }\quad
  \Oo_{F_1}(\alpha h+\beta f)\cong \Oo _X(\alpha ,\beta-\alpha),$$ 
 for any $a,b\in\ZZ$ and  for any $\alpha,\beta\in\ZZ$.

%
%

\begin{lemma}\label{h0X}
Let $X$ be a surface of bidegree $(1,0)$ and $Y$ be a surface of bidegree $(0,1)$. For any $a\ge 0, b\ge0$, we have
$$h^0(\Oo_X(a,b))=a(b+1)+
\binom{b+2}{2},\quad
h^0(\Oo_Y(a,b))=b(a+1)+
\binom{a+2}{2},$$
and
$$h^1(\Oo_X(a,b))=h^1(\Oo_Y(a,b))=0.$$ 
\end{lemma}
\begin{proof}
It is easy to prove, as in \cite[Lemma 2.3]{ABBS}, that $h^1(\Oo _\FF(a,b))=0$ if
 $a\ge -1$ and $b\ge -1$.
Hence, by using the exact sequence
$$0 \to \Oo_\FF(a-1,b)\to \Oo_\FF(a,b)\to \Oo_X(a,b)\to 0,$$
we conclude. 
\end{proof}

{\begin{remark}\label{very ample}
Given $X$ (resp.\ $Y$) a surface of bidegree $(1,0)$ (resp.\ $(0,1)$), we know by \cite[Corollary 2.18]{Hartshorne} that the line bundles $\Oo_X(a,b)$ and $\Oo_Y(a,b)$ are very ample, if $a>0,b>0$. 
\end{remark}}

{
We know that, if $L$ is a fiber of the ruling of $X$ (that is in the class of $f$), then $L\in|\Oo_X(1,0)|$ hence we have
$$0 \to \Oo_X(a-1,b)\to \Oo_X(a,b)\to \Oo_L(a,b)\to 0$$
and by the previous lemma we get $\Oo_L(a,b)=\Oo_{\PP^1}(b)$.
Hence a curve $L$ in the class of $f$ has bidegree $(0,1)$.}

{On the other hand a curve $L\in|\Oo_X(0,1)|$ in the class of $h+f$ is such that
$$0 \to \Oo_X(a,b-1)\to \Oo_X(a,b)\to \Oo_L(a,b)\to 0$$
and by the previous lemma we get $\Oo_L(a,b)=\Oo_{\PP^1}(a+b)$,
hence a curve $L$ in the class of $h+f$ has bidegree $(1,1)$.
}




\section{A Miyaoka-type bound for smooth surfaces}\label{bound}


In this section we aim to obtain a Miyaoka-type bound \cite{miyaoka} in the case of a smooth surfaces of bidegree $(a,b)$
contained in the flag manifold $\FF$. 
In order to obtain such a result we need to study the Chow ring of the flag and compute Chern classes. 
In particular, applying \cite[eq. (6)]{miyaoka}, we will prove the following result.
\begin{theorem}\label{s1}
Let $S\subset \FF$ be a smooth surface of bidegree $(a,b)$, with $a\ge 3,b\ge3$. Then $S$
contains at most 
\begin{equation}\label{eqab}
\frac{2 (a + b - 2) (3 a^2 b - a^2 + 3 a b^2 - 4 a b + 3 a - b^2 + 3 b)}{(a + b - 1)^2}
\end{equation}
disjoint smooth conics. Moreover, this bound also holds for twistor fibers.
\end{theorem}

We recall some known facts about Chern classes, as a general reference see~\cite{Fulton}.
For any exact sequence of vector bundles on any algebraic scheme $Y$:
$$0 \to E'\to E\to E''\to 0$$
we have $c_1(E) =c_1(E')+c_1(E'')$ and $c_2(E) = c_2(E')+c_2(E'') +c_1(E')\cdot c_1(E'')$, where
$c_i(F)$ are the Chern classes of the vector bundle $F$ and the sum and the product is the sum and the product in the Chow
ring of $Y$. 

Now considering the product $\PP^2\times\PP^2$, we have
$$c_1(\Oo _{\mathbb{P}^{2}\times \mathbb{P}^{2}}(1,0))=L\times \PP^2, \qquad\quad
c_1(\Oo _{\mathbb{P}^{2}\times \mathbb{P}^{2}}(0,1))=\PP^2\times R,$$
where $L$ is a line in the first $\PP^{2}$ and $R$ a line in the second one.

Recall that 
$c_1(T\PP^2)=3\ell$
and $c_2(T\PP^2)=3p$ where $\ell$ the class of a line and $p$ is the class of a point in the Chow ring of  $\PP^2$, see 
e.g.\ \cite[Example 3.2.11]{Fulton}.
Since $T(\PP^2\times\PP^2)\cong \Pi_1^*(T\PP^2)\oplus \Pi_2^*(T\PP^2)$, we have
$$c_1(T(\PP^2\times \PP^2))=3(L\times \PP^2)+3(\PP^2\times R)=c_1(\Oo_{\PP^2\times\PP^2}(3,3))$$
and 
\begin{equation}\label{eqc2*}
c_2(T(\PP^2\times \PP^2))=3(\{o\}\times \PP^2)+3(\PP^2\times \{o'\})+9 (L\times R),
\end{equation}
where $o$ and $o'$ are points in the two planes.

Since the tangent bundle of a variety is the dual of the cotangent bundle,
we have that
$$c_1(T\FF)\cong-c_1(\omega_\FF)\cong-c_1(\Oo _{\FF}(-2,-2))\cong c_1(\Oo _{\FF}(2,2)).$$
In the following remark we describe the second Chern class.

\begin{remark}\label{classi flag}
From
 the normal bundle sequence
$$0\to T\FF \to T(\mathbb{P}^{2}\times\mathbb{P}^{2})_{|\FF}\to \Oo _{\FF}(1,1)\to 0$$ 
we have that
$$c_1(T\FF) =c_{1}(T(\mathbb{P}^{2}\times
\mathbb{P}^{2})_{|\FF})-c_{1}(\Oo _{\FF}(1,1))=c_{1}(\Oo _{\FF}(3,3))-c_{1}(\Oo _{\FF}(1,1))
 =c_1(\Oo _{\FF}(2,2))$$ 
and that
 $$c_2(T\FF) =c_2(T(\mathbb{P}^{2}\times\mathbb{P}^{2})_{|\FF}) -c_1(\Oo _{\FF}(2,2))\cdot c_1(\Oo _{\FF}(1,1)).$$
Now from Formula~\eqref{eqc2*} 
we get that $$c_2(T(\mathbb{P}^{2}\times
\mathbb{P}^{2})_{|\FF})=3(\{o\}\times R) + 3(L\times \{o'\})+9C,$$ where $C\subset \FF$
is a smooth conic.  
Indeed, the class
represented by $L\times R$ is embedded in $\mathbb{P}^{8}$ by the Segre embedding 
as a smooth quadric surface spanning a $3$-dimensional linear space. Remember that $\FF$ is a hyperplane
section of the Segre variety $\Sigma$, hence the cycle $L\times R$ restricted to the flag is a smooth conic. 
\end{remark}

From \cite[Proposition 3.11]{ABBS} we immediately get the following.
\begin{lemma}\label{prop-part-a}
Let $S\subset\FF$ be a smooth surface of bidegree $(a,b)$. Then we have
\begin{enumerate}
\item $c_1(\Oo _S(1,0))\cdot c_1(\Oo _S(1,0)) =b$;
\item $c_1(\Oo _S(0,1))\cdot c_1(\Oo _S(0,1)) =a$;
\item $c_1(\Oo _S(1,0))\cdot c_1(\Oo _S(0,1)) =c_1(\Oo _S(1,0))\cdot c_1(\Oo _S(0,1)) =a+b$.
\end{enumerate}
\end{lemma}

\begin{remark} \label{c1S}
If $S\in |\Oo _{\FF}(a,b)|$, from Formula~\eqref{omegaF}, the adjunction formula gives $\omega _S\cong \Oo _S(a-2,b-2)$, which implies
\begin{equation}\label{c1tan}
c_1(TS) = c_1(\Oo _S(2-a,2-b)).\end{equation}
Hence, from the previous lemma it is possible to derive the formula for $c_{1}^{2}(S)$, already given in~\cite[Remark 3.12]{ABBS}:
$$
c_{1}^{2}(S)=3a^{2}b+3ab^{2}-4a^{2}-4b^{2}-16ab+12a+12b.
$$

\end{remark}

We now pass to compute the second Chern class of a surface $S$.

\begin{proposition} \label{c2S}
Let $S\subset\FF$ be a smooth surface of bidegree $(a,b)$. Then we have
$$c_2(S) = 6a+6b +3 a^2 b - 2 a^2 + 3 a b^2 - 8 a b - 2 b^2.$$
\end{proposition}
\begin{proof}
From the normal bundle sequence of $S$ in $\FF$
$$0 \to TS \to T\FF_{|S} \to
\Oo_S(a,b)\to 0,$$
we have, by using \eqref{c1tan}, $$c_2(TS) = c_2(T\FF_{|S})-c_1(\Oo _S(a-2,b-2))\cdot c_1(\Oo _S(a,b)).$$
By means of Lemma~\ref{prop-part-a}, we compute
\begin{align*}
c_1(\Oo _S(a-2,b-2))\cdot c_1(\Oo _S(a,b))
=&(a-2)ab+(a-2)b(a+b)\\
&+(b-2)a(a+b)+(b-2)ba\\
=&3 a^2 b - 2 a^2 + 3 a b^2 - 8 a b - 2 b^2.
\end{align*}

Now we compute $c_2(T\FF_{|S})$, using Remark \ref{classi flag}.
Thanks to Lemma~\ref{prop-part-a} we have that $c_1(\Oo _S(1,1))\cdot c_1(\Oo _S(2,2)) = 2a+2b+4(a+b) =6(b_1+b_2)$.
Hence we get $$c_2(T\FF_{|S})=3b+3a +9(a+b)-6(a+b) =6a+6b.$$
Therefore, we obtain the formula in the statement.
\end{proof}

We are now ready to prove the main result of this section.

\begin{proof}[Proof of Theorem~\ref{s1}]
Recall that $\omega _S \cong \Oo _S(a-2,b-2)$. Since $a\ge3$ and $b\ge3$, then $S$ is of general type. 

We prove now that $S$ is minimal.
Let $C\subset S$ be any smooth rational curve of bidegree $(d_1,d_2)\in \NN^2\setminus\{(0,0)\}$. The adjunction formula gives $\omega_C \cong \omega_S(C)_{|C}$, that is, with another notation, 
$K_C = C\cdot (C+K_S)$. 
Hence, since $C$ is rational we have that $\omega_C=\Oo_C(-2)$
and by using Remark~\ref{curved1d2}, we 
compute $$-2=C^2 +d_1(a-2) +d_2(b-2).$$  Thus $C^2 \ne -1$. Hence $S$ is a minimal surface of general type and we
are in the hypothesis of~\cite[Proposition 2.1.1]{miyaoka}.

Assume that $S$ contains $k$ pairwise disjoint smooth conics. 
Since a smooth conic has bidegree $(1,1)$, then it has self-intersection $2-a-b<0$.
Then, by applying \cite[eq. (6)]{miyaoka} we get
\begin{equation*}\label{eqs1}
k\frac{(a+b-1)^2}{3(a+b-2)} \le c_2(S) -\frac{1}{3}c_1^2(S),
\end{equation*}
and hence
\begin{equation*}\label{eqs2}
k \le \frac{3(a+b-2)}{(a+b-1)^2}\left(c_2(S) -\frac{1}{3}c_1^2(S)\right).
\end{equation*}

As we have already computed $c_{1}^{2}(S)$ in Remark \ref{c1S} and $c_{2}(S)$ in Proposition \ref{c2S},
we have that
\begin{align*}
k\le &\frac{3(a+b-2)}{(a+b-1)^2}\left(c_2(S) -\frac{1}{3}c_1^2(S)\right)\\
=&\frac{3(a+b-2)}{(a+b-1)^2}\left(6a+6b +3 a^2 b - 2 a^2 + 3 a b^2 - 8 a b - 2 b^2\right.\\
&\left. -\frac{1}{3}(3a^{2}b+3ab^{2}-4a^{2}-4b^{2}-16ab+12a+12b)\right)\\
=&\frac{2 (a + b - 2) (3 a^2 b - a^2 + 3 a b^2 - 4 a b + 3 a - b^2 + 3 b)}{(a + b - 1)^2}.
\end{align*}

Finally, as two different twistor fibers are disjoint, and $\Tt\subset \Cc$
the same inequality holds also for twistor fibers.
\end{proof}

\begin{remark}
Notice that when $a=b$, Formula~\eqref{eqab} becomes
$$
k\le\frac{24( a^2 -  a +  1) ( a - 1) a}{(2 a - 1)^2}.
$$
This quantity is clearly bigger than $3a^{2}(=a^{2}+ab+b^{2})$ obtained in~\cite[Proposition 8.1]{ABBS} {as a bound for twistor lines contained in} smooth  surfaces {which are not $j$-invariant}. Notice that in Theorem \ref{s1} we are only assuming that the surface is smooth. 

\begin{figure}[h]
\vspace{-5pt}
\includegraphics[scale=0.7]{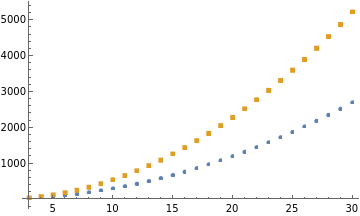}
\vspace{-5pt}
\caption{In blue round markers the function $3a^{2}$, while in yellow square markers $\frac{24( a^2 -  a +  1) ( a - 1) a}{(2 a - 1)^2}$, for $3\le a\le 30$.}
\end{figure}

Moreover, in Proposition~\ref{a0} below, we will prove that if an integral  surface containing $x\ge a^{2}+ab+b^{2}$ smooth conics exists, then it is unique.
\end{remark}

\begin{remark}\label{3.7}
Along the same lines of the proof of Theorem~\ref{s1}, it is possible to obtain an analogous upper bound
on the number of curves of bidegree $(1,0)$. In particular, if $S$ is a surface of bidegree $(a,b)$, with $a\ge 3, b\ge 3$, and $C\subset S$ is a curve of bidegree $(1,0)$, then $C^{2}=-a$. Hence, the maximum number $k$ of $(1,0)$ curves contained in $S$ is equal to
\begin{align*}
k \le& \frac{3a}{(a+1)^2}\left(c_2(S) -\frac{1}{3}c_1^2(S)\right)\\
=&\frac{3a}{(a+1)^2}\left(6a+6b +3 a^2 b - 2 a^2 + 3 a b^2 - 8 a b - 2 b^2\right.\\
&\left. -\frac{1}{3}(3a^{2}b+3ab^{2}-4a^{2}-4b^{2}-16ab+12a+12b)\right)\\
=&\frac{2 a ( a^2 ( 3 b-1) + a (3 b^2-4b+3)-(b-3) b )}{(1 + a)^2}.
\end{align*}

\end{remark}

\section{Surfaces with prescribed number of conics}\label{prescribed}
%
%

In this section we prove two results. First we show that, if there exists a surface in $\FF$ containing a certain amount of smooth conics, then this
surface is unique. Later, we will prove that, under certain numerical hypotheses, a general surface containing a general set of disjoint smooth
conics, does not contains any other conic disjoint from the previous one and it is irreducible and smooth.

In~\cite[Proposition 8.1]{ABBS} we proved that an integral surface $S$ of bidegree $(a,b)$ that is not $j$-invariant,
is such that $S\cap j(S)$ contains at most $a^{2}+ab+b^{2}$ smooth conics. We now prove that this number 
is critical in order to find at most one integral surface containing that particular number of pairwise disjoint smooth conics.

\begin{proposition}\label{a0}
Fix positive integers $a,b$. Let $T\subset \FF$ be any union of 
$$x\ge a^2+ab+b^2$$
pairwise disjoints smooth conics. Then there exists at most a unique integral surface of bidegree $(a,b)$ which contains $T$.
\end{proposition}

\begin{proof}
Assume by contradiction the existence of two integral $S, S'\in |\Oo_{\FF}(a,b)|$ such that $S\ne S'$. Since $S$ is integral, the intersection $S\cap S'$ is scheme-theoretically an effective Cartier divisor of $S$
whose degree 
is $a^2+ab+b^2$. Thus $T =S\cap S'$. Since $S\cap S'$ is an ample Cartier divisor of an integral variety of positive dimension, $S\cap S'$ is connected. But $T$
has $a^2+ab+b^2>1$ connected components.
\end{proof}

\begin{remark}
The connectedness of $S\cap S'$ was used also in~\cite[Corollary]{ABBS},
for $S'=j(S)$ in order to prove that an integral surface of bidegree $(a,b)$
contains at most $a^{2}+ab+b^{2}-1$ twistor fibers.
\end{remark}

We introduce now some notation which we will use in the next proofs.
For all $(e,f,g)\in \NN^3\setminus \{(0,0,0)\}$ let $\Cc(e,f,g)$ denote the set of all $(e+f+g)$-ples $(C_1,\dots ,C_e,F_1,\dots ,F_f,G_1,\dots,G_g)$ such that each $C_i$ is a smooth conic, each $F_i$ is a curve of bidegree $(1,0)$, each $G_i$ is a curve of bidegree $(0,1)$
and the union $C_1\cup \cdots \cup C_e\cup F_1\cup \cdots \cup F_f\cup G_1\cup \cdots \cup G_g$ has $e+f+g$ connected disjoint components.
Set $\Cc(e):= \Cc(e,0,0)$.


We will denote by $\Ii_{X,Y}$ the sheaf ideal of a scheme
 $X$ contained in a projective variety $Y$.

\begin{remark}\label{facile}
Let $a,b\ge0$ and $T,C\subset \FF$ be two curves such that $C$ is a smooth conic and  $C\cap T=\emptyset$.
Then from the exact sequence 
\begin{equation} \label{ex-seq}
0 \to \Ii_{C\cup T,\FF}(a,b)\to \Ii_{T,\FF}(a,b)\to \Oo_C(a+b)\to 0,
\end{equation}
since $C$ is rational, 
we get $h^1(\Ii_{T,\FF}(a,b))\le h^1(\Ii_{C\cup T,\FF}(a,b))$
\end{remark}

\begin{theorem}\label{c1-prima parte}
Fix integers $b\ge a\ge1$, $0\le  x\le a(a-1)/2$. Let $T\subset \FF$ be a general union of $x$ smooth conics.
Then
$h^1(\Ii _{T,\FF}(a,b)) =0$.
\end{theorem}
\begin{proof}
Thanks to Remark~\ref{facile}, it is sufficient to prove the statement for the case $x=a(a-1)/2$.
We first apply induction on $h=a-b$. If $h=0$  we need to prove that
\begin{equation}\label{a,a}
h^1(\Ii _{T,\FF}(a,a)) =0.\end{equation}
To prove Formula~\eqref{a,a}, we use induction on the integer $a$, the case $a=1$ being trivial since $x=0$. 
Take $a\ge1 $ and assume that  \eqref{a,a} holds for the integer $a$. Take a general $E\in \Cc(x)$.
Let $X$ be a {general} surface of bidegree $(1,0)$ and recall that it is birational to a Hirzebruch surface $F_1$, see Section~\ref{preliminaries}.

By Remark~\ref{intersezioni}, we get that $X\cap E$ is formed by $x$ general points of $X$. 
Take a general union $G\subset X$ of $a$ fibers of the ruling of $X$, which are curves of bidegree $(0,1)$.
Restricting now to the surface $Y$, we get as a trace
$(E\cup G)\cap Y=(E\cap Y)\cup G$
and as a residue
$\Res_Y(E\cup G)=E$, hence we get
the residual exact sequence:
\begin{equation}\label{eqc1}
0\to \Ii_{E,\FF}(a,a)\to \Ii _{E\cup G, \FF}(a+1,a)\to \Ii_{(E\cap X)\cup G,X}(a+1,a)\to 0.
\end{equation}

By using Formula~\eqref{coeff}, we have
$$\Oo _X(a+1,a) \cong \Oo_{F_1}(ah+(2a+1)f)$$ and 
 $$\Ii _{G,X}(a+1,a) \cong \Oo_{F_1}(ah+(a+1)f)\cong \Oo _X(1,a).$$  
 Thus, by Lemma \ref{h0X}, $h^1(\Ii _{G,X}(a+1,a))=0$
and 
$h^0(\Ii _{G,X}(a+1,a))\ge a(a-1)/2$. Hence, since the points in $E\cap X$ are general we get
$h^1(\Ii_{(E\cap X)\cup G,X}(a+1,a))=0$.

Therefore,
from
the exact sequence \eqref{eqc1} and
 by induction, we obtain $h^1(\Ii _{E\cup G}(a+1,a)) =0$.

{Consider now a {general} surface $Y$ of bidegree $(0,1)$, which is again birational to a Hirzebruch surface $F_1$.}
Recall that $X\cap Y$ is a smooth conic $C$ of bidegree $(1,1)$. In the identification of $Y$ with $F_1$ the curve $C\subset Y$ corresponds to a smooth element  of $|\Oo_{F_1}(h+f)|$,
and the set $G\cap Y$ is formed by $a$ general points of $C$. Let $F\subset Y$ be the union of the $a$ curves of bidegree $(1,0)$ containing one of the points of $G\cap Y$ and set $U:= E\cup G\cup F$. 
{The set $U$ belongs to the closure of $\Cc(x,a,a)$, because any component of $G$ meets one component of $F$. Moreover, thanks to Remark~\ref{c5}}, it is the flat limit of a flat family in $\Cc(x+a)$.
Hence, by semicontinuity,  to prove \eqref{a,a} for the integer $a+1$ it is sufficient to prove that $h^1(\Ii _{U,\FF}(a+1,a+1))=0$.

Restricting now to the surface $Y$, we get as a trace
$U\cap Y=(E\cap Y)\cup F$
and as a residue
$\Res_Y(U)=E\cup G$, hence we get
the following exact sequence:
\begin{equation}\label{eqc2}
0 \to \Ii_{E\cup G,\FF}(a+1,a)\to \Ii _{U,\FF}(a+1,a+1)\to \Ii _{(E\cap Y)\cup F,Y}(a+1,a+1)\to 0.
\end{equation}

Thus, again, up to identification of $Y$ with the Hirzebruch surface $F_1$, we have:
$$\Oo _Y(a+1,a+1) \cong \Oo_{F_1}((a+1)h+(2a+2)f)$$ and 
 $$\Ii _{F,Y}(a+1,a+1) \cong \Oo_{F_1}((a+1)h+(a+2)f)\cong \Oo _Y(a+1,1).$$ 
 Thus, by  Lemma \ref{h0X}, 
 we have $h^1(\Ii _{F,Y}(a+1,a+1))=0$
and 
$h^0(\Ii _{F,Y}(a+1,a+1))\ge  a(a-1)/2$. Then, since the points in $E\cap Y$ are general, we get
$h^1(\Ii _{(E\cap Y)\cup F,Y}(a+1,a+1))=0$, and we conclude by induction that \eqref{a,a} holds for any $a\ge1$.

Consider now $h=b-a\ge1$ and assume by induction that $h^1(\Ii_{T,\FF}(a,b))=0$, where $T$ is a general union of $x$ smooth conics.
Take
$$0\to \Ii_{T,\FF}(a,b)\to \Ii _{T, \FF}(a,b+1)\to \Ii_{T\cap Y,Y}(a,b+1)\to 0,$$
where $Y$ is a general surface of bidegree $(0,1)$.
Since $T\cap Y$ are $x$ general points we immediately have $h^1(\Ii_{T\cap Y,Y}(a,b+1))=0$ and we conclude.
\end{proof}

Thanks to the previous vanishing theorem, we are now able to prove the qualitative result on surfaces containing a prescribed numbers of conics
stated in Theorem~\ref{c1-seconda parte}.


\begin{proof}[Proof of Theorem~\ref{c1-seconda parte}]
We start by proving part (i).
Let $x\le (a-1)(a-2)/2$ and $T$ be a general union of $x$ smooth conics. 
Given a general $S\in |\Ii_{T,\FF}(a,b)|$, 
we want to prove that $S$ contains no other smooth conic $C$ such that $C\cap T=\emptyset$.

Assume the existence of such a  $C$ (and hence $C$ is not generic),
and take  a general surface $X$ of bidegree $(1,0)$ containing $C$. Up to the identification of $X$ with the Hirzebruch surface $F_1$, $C$ is an element of the linear sistem $|h+f|$. Since any two elements of $|h+f|$ meets, $X$ contains no connected component of $T$. Thus the residual exact sequence of $X$ gives the following exact sequence
\begin{equation}\label{eqccc1}
0 \to \Ii_{T,\FF}(a-1,b)\to \Ii_{T\cup C,\FF}(a,b)\to \Ii _{C\cup (T\cap X),X}(a,b)\to 0.
\end{equation}
By Theorem \ref{c1-prima parte}, we know that $h^1(\Ii_T(a-1,b))=0$. 

Up to the identification of $X$ with $F_1$, we have $C\in |h+f|$ and $$\Oo_X(a,b)\cong \Oo_{F_1}(bh+(a+b)f).$$ 
Hence, by \eqref{coeff},
$$\Ii _{C,X}(a,b)\cong \Oo_{F_1}((b-1)h+(a+b-1)f)\cong \Oo_X(a,b-1),$$
and {so $\Ii _{C\cup (T\cap X),X}(a,b)=\Ii _{(T\cap X),X}(a,b-1)$.
Since $T\cap X$ are $x$ general points and $x\le h^0(\Oo_X(a,b-1))$, by Lemma \ref{h0X} we conclude 
 that $h^1(X,\Ii _{T\cap X,X}(a,b-1))=0$. }


 Thus we have proved that
$$h^1(X, \Ii _{C\cup (T\cap X),X}(a,b))=0,$$ 
and therefore, from sequence \eqref{eqccc1}
we get $h^1(\Ii_{T\cup C}(a,b))=0$. 
By the exact sequence \eqref{ex-seq}, we also have
$$h^0(\Ii _{T\cup C}(a,b)) =h^0(\Ii _T(a,b))-(a+b+1).$$
{Recall that $\dim \Cc(1) =4$ and let $\widetilde\Cc(1)$ be the open subset of the smooth conics disjoint from $T$. 
To any $C \in \widetilde\Cc(1)$ we can associate
a subspace $|\Ii_{T\cup C}(a,b)|\subseteq |\Ii_T(a,b)|$ of codimension
$a+b+1$.
The union of all these subspaces has dimension less or equal than $\dim|\Ii_{T\cup C}(a,b)|+4< \dim|\Ii_{T}(a,b)|$,
since $a+b+1>4$.
Therefore,
a general $S\in |\Ii_T(a,b)|$ cannot contain any $C\in \widetilde\Cc$ and this contradicts our assumption.
}
\smallskip

We now prove part (ii).
By Theorem \ref{c1-prima parte}, we know that $h^1(\Ii _T(a-1,b-1)) =0$.
Using the exact sequence 
$$0\to \Ii_{T,\FF}(c,d)\to \Oo_{\FF}(c,d)\to \Oo_T(c,d) \to 0,$$
for $c=a-2, d=b-2$, we have that $h^2(\Ii_{T,\FF}(a-2,b-2)) =0$. Analogously we get $h^i(\Ii_{T,\FF}(a-i,b-i)) =0$ for $i\ge3$.
Hence by the Castelnuovo-Mumford's lemma $\Ii_{T,\FF}(a,b)$ is globally generated. Therefore, by Bertini's theorem, a general $X\in |\Ii_{T,\FF}(a,b)|$ is integral and smooth outside $T$. 

To conclude, it is sufficient to prove that for each $p\in T$
the set of all $X\in |\Ii_{T,\FF}(a,b)|$ singular at $p$ has codimension at least $2$ in $|\Ii_{T,\FF}(a,b)|$, i.e. 
$$h^0(\Ii _{T\cup 2p,\FF}(a,b)) \le h^0(\Ii _{T,\FF}(a,b))-2,$$
where we denote by $2p$ the $0$-dimensional scheme of $\FF$ defined by the ideal $\Ii_{p,\FF}^2$.
{Indeed in this case, since $\dim T=1$, the singular surfaces in $|\Ii _{T,\FF}(a,b)|$ are parametrized by a one-dimensional family of subspaces of codimension at least $2$, hence a general surface is smooth.}

Consider the exact sequence
$$0\to \Ii _{T\cup 2p,\FF}(a,b)\to \Ii _{T,\FF}(a,b)\to \Ii_{T,\FF}\otimes \Oo_{2p}(a,b)\to 0.$$
Since $\deg (2p) =4$ and $\deg (2p\cap T)=2$, {we have $h^0(\Ii_{T,\FF}\otimes \Oo_{2p}(a,b))\ge2$.}
Hence it is sufficient to prove that
$$h^1(\Ii _{T\cup 2p,\FF}(a,b)) =0.$$
 
Fix $p\in T$ and let $C$ be the connected component of $T$ containing $p$. Set the following notation $E:= T\setminus C$. By Theorem \ref{c1-prima parte} and sequence \eqref{ex-seq}, we have  $h^0(\Ii_{T,\FF}(a,b-1)) =h^0(\Ii_{E,\FF}(a,b-1)) -a-b$.
Thus $p$ is not in the base locus of $|\Ii_{E,\FF}(a,b-1)|$. 
Fix $S\in |\Ii_{E,\FF}(a,b-1)|$ such that $p\notin S$. 
Let $Y$ be surface of $|\Oo_{\FF}(0,1)|$ containing $C$ and consider the residual exact sequence with respect to $Y$:
\begin{equation}\label{eqcc1}
0 \to \Ii_{E\cup p,\FF}(a,b-1) \to \Ii _{T\cup 2p}(a,b)\to \Ii_{{(E\cap Y)}\cup C\cup (2p\cap Y),Y}(a,b)\to 0.
\end{equation}

Now we prove that 
\begin{equation}\label{1vanish}
h^1(\Ii _{E\cup p,\FF}(a,b-1)) =0.
\end{equation}

{Recall that $T=E\cup C$ and $p\in C$, hence we have the exact sequence
$$0 \to \Ii_{T,\FF}(a,b-1) \to \Ii _{E\cup p}(a,b-1)\to \Ii_{p,C}(a+b-1)\to 0.$$
Thanks to Theorem~\ref{c1-prima parte} we have that  $h^1(\Ii _{T,\FF}(a,b-1)) =0$; on the other hand,  since $C$ is a smooth rational curve, we have
$h^1(\Ii_{p,C}(a+b-1))=h^1(\Oo_{C}(a+b-2))=0$ and this proves \eqref{1vanish}.
}

%
%
%

In order to conclude it is sufficient to prove now that
\begin{equation}\label{2vanish}
\Ii_{(E\cap Y)\cup C\cup (2p\cap Y),Y}(a,b).
\end{equation}

{
Note that 
$\Ii_{C,Y}(a,b)\cong \Oo_{F_1}((a-1)h+(a+b-1)f)\cong \Oo_Y(a-1,b), $
hence, by Remark \ref{very ample}, we know that $\Ii_{C,Y}(a,b)$ is very ample.
Therefore we get
$h^1(\Ii_{C\cup (2p\cap Y),Y}(a,b))=0$. Since $E\cap Y$ are $x-1$ general points we conclude that \eqref{2vanish}
 holds.
}

%
%
Thus the exact sequence \eqref{eqcc1} gives $h^1(\Ii _{T\cup 2p,\FF}(a,b))=0$, concluding the proof.
\end{proof}

\medskip

We now pass to apply the previous result to twistor fibers. We denote now by 
$$\Tt(k)\subset \Cc(k)$$ 
the set of all $(C_1,\dots ,C_k)$
such that each $C_i$ is a twistor fiber.
The action of the anti-holomorphic involution $j$, defined in Formula~\eqref{mapj}, on $\Cc(k)$ is given by $j((C_1,\dots ,C_k))=(j(C_1),\dots, j(C_k))$, and $\Tt(k)$ is the fixed locus with respect to this action.

We need to introduce the notion of {\it general} union of $k$ twistor fibers.
\begin{definition}
%
Following~\cite[Definition 3.1]{altavillaballico1}, we say that a property $\Pp$ is true for a \textit{general union} of $k$ twistor fibers or that $k$ \textit{general twistor fibers} satisfy $\Pp$
if there is a non-empty Zariski open subset $U$ of $\Cc(k)$ such that  $\Pp$ is true for all 
$(C_1,\dots ,C_k)\in U\cap \Tt(k)$,
{such that $C_{m}\neq C_{n}$ for any $m\neq n$.}
\end{definition}

In complete analogy with~\cite[Lemma 3.2]{altavillaballico1} we are able to state that $\Tt(k)$ is a Zariski dense of $\Cc(k)$. This observation, Theorem~\ref{c1-seconda parte} and the
fact that two twistor fibers are disjoint, lead us to the following result.

%
%

\begin{corollary}
Fix integers $b\ge a\ge2$, $0\le x\le (a-1)(a-2)/2$. The 
following statements hold true.
\begin{enumerate}
\item If $T\subset \FF$ is a general union of $x$ twistor fibers, then
 a general $S\in |\Ii_{T,\FF}(a,b)|$  is irreducible, smooth and contains exactly $x$ twistor fibers.
\item There exists a smooth irreducible surface $S$ of bidegree $(a,b)$ containing exactly $x$ twistor fibers.
\item Near $S$, the family of all the surfaces of bidegree $(a,b)$ is a smooth differentiable manifold of real dimension 
$2\left(\binom{a+2}{2}\binom{b+2}{2}-\binom{a+1}{2}\binom{b+1}{2} -x(a+b+1)\right) +4x.$
\end{enumerate}
\end{corollary}

\begin{proof}
{Recall that the set $\Tt(k)$ is a connected and smooth real analytic variety of real dimension $4k$ which is Zariski dense in $\Cc(k)$.}
By Theorem \ref{c1-seconda parte} and semicontinuity, we get our statements.
\end{proof}
%


\section{Surfaces with infinitely many twistor fibers}\label{infinite-sec}

The last results of the previous section introduce us to to the  study of surfaces containing infinitely elements
of $\Tt$.
 In particular, here we aim to describe a method to construct such kind of surfaces. Moreover, this result will also give
the existence of infinitely many examples.

We start with the following remark.


\begin{remark}\label{f1}
Let $S\subset \FF$ be an integral surface.
If $S$ contains infinitely many twistor fibers, then there is a positive dimensional family of
twistor fibers contained in $S$.
In fact, let $G\subset \Cc$ be any infinite set and let $\overline{G}$ be its Zariski closure in $\Cc$.
Since $G$ is infinite, $\overline{G}$ has a positive-dimensional component.  If every element of $G$ is contained in $S$, then every element of $\overline{G}$ is contained in $S$.

In particular any desingularization $S'$ of $S$ is uniruled, i.e. it has Kodaira dimension $\kappa (S') =-\8$.
In particular $\kappa (S) =-\8$ if $S$ contains infinitely many twistor fibers.
Notice that the same argument can be adapted to the case of surfaces in $\mathbb{P}^{3}$, seen as twistor space of $\mathbb{S}^{4}$, containing infinitely many twistor fibers \cite{altavillaballico3}.
\end{remark}

We now show that surfaces with isolated singularities containing infinitely many twistor fibers must have bidegree $(1,1)$.

\begin{proposition}\label{f4}
Let $S\in |\Oo _{\FF}(a,b)|$ be an integral surface such that $\mathrm{Sing}(S)$ is finite and $S$ contains infinitely many twistor fibers. Then $a=b=1$ and $S$ is smooth.
\end{proposition}

\begin{proof}
Since $S$ has infinitely many twistor fibers, thanks to~\cite[Corollary 8.7]{ABBS} we have $j(S)=S$ and hence $a=b$
By Remark \ref{f1} there is a positive dimensional family $\mathcal{A}$ of
twistor fibers contained in $S$.


Recall that since $\omega _{\FF} \cong \Oo _{\FF}(-2,-2)$, the adjunction formula gives $\omega _S\cong \Oo
_S(a-2,a-2)$. 
 Since any two different twistor fibers are disjoint and
$\mathrm{Sing}(S)$ is finite, there is a twistor fiber $D\in \mathcal{A}$ such that $D\cap \mathrm{Sing}(S) =\emptyset$. Since $D$ is
contained in $\mathrm{Sm}(S)$, the smooth locus of $S$, the normal bundle sequence 
$$0 \to TD \to T(\mathrm{Sm}(S))_{|D} \to
N_D\to 0,$$
gives $\omega
_D\cong \Oo
_D(a-2,a-2)\otimes N_{D}$, where $N_D$ is the normal bundle of $D$ in $\mathrm{Sm}(S)$.

Since we may take as $D$
a general element of a positive-dimensional family of curves of $\mathrm{Sm}(S)$, the line bundle  $N_{D}$
has degree at least $0$. Since $D\cong \PP^1$, we have $\deg (\omega _D)=-2$. Thus $a<2$ and $S$ is a $j$-invariant  $(1,1)$ surface.  Finally, by \cite[Lemma 4.1]{ABBS} we conclude that $S$ is smooth.
\end{proof}

Hence, in order to find other examples of surfaces in $\FF$ containing infinitely many twistor fibers, we must look at  surfaces which have singular locus of positive dimension. In the case of surfaces of bidegree $(1,1)$ the analysis is
quite simple.
We recall that a reducible $(1,1)$-surfaces is the union of two (Hirzebruch) surfaces of bidegrees $(1,0)$ and $(0,1)$, respectively, intersecting on a smooth conic (see~\cite[Section 3.2 and Proposition 4.5]{ABBS}). In the next remark,
we show that there are not reducible surfaces of bidegree $(1,1)$ containing infinitely many twistor fibers.

\begin{remark}
Let $S$ be a reducible surface of bidegree $(1,1)$. Write $S = X\cup Y$ with $X\in |\Oo _{\FF}(1,0)|$ and $Y\in |\Oo _{\FF}(0,1)|$. Any irreducible curve $T\subset S$
is contained either in $X$ or in $Y$. We have $j(X)\in |\Oo _{\FF}(0,1)|$ and $j(Y)\in |\Oo _{\FF}(1,0)|$. Thus if $j(T)=T$, then either $T\subseteq j(X)\cap X$ or $T\subseteq j(Y)\cap Y$. 
Since $j(X)\cap X$ and $j(Y)\cap Y$ have dimension $1$, $X$ contains at most one twistor fiber. This case happens exactly when $S=j(S)$.
\end{remark}

From Proposition~\ref{f4} we know that if $a\geq 2$ any $(a,a)$-surface containing infinitely many twistor fibers  has singular locus containing a curve. We are now ready to prove the main theorem of this section, describing an explicit construction of such a surface, for any $a\ge2$ .

\begin{theorem}\label{infinite}
Let $C$ be a smooth and connected complex projective curve {of genus $g$} defined over $\RR$ and with $C_\RR \ne \emptyset$. {For any $a\ge g+2$,} 
 there exists an integral surface $S\in |\Oo_{\FF}(a,a)|$ containing
infinitely many twistor fibers and birational to $C\times \CC\PP^1$.
\end{theorem}

\begin{proof}

 By Remark \ref{f1} any surface $S\subset \FF$ containing infinitely many twistor fibers is birationally ruled by conics and
hence it is birational over $\CC$ to $C\times \CC\PP^1$ for some smooth projective curve $C$.

Fix a general non-special 
very ample line bundle $\mathcal L$ of degree $a$ on $C$ defined over $\RR$. 
{Since $a\ge g+2$, then $h^0(C,\mathcal L) \ge {3}$. }
Take a general $3$-dimensional real subspace $V_\RR\subseteq H^0(C,\Ll)_\RR$ and let $V=V_\RR\otimes_\RR \CC \subseteq H^0(C,\Ll)=H^0(C_\RR\otimes _{\RR}\CC,\Ll)$ (where the last equality can be found for instance in~\cite[Chapter III, Proposition 9.3 and 9.4]{Hartshorne}). 
Thus $V$ induces a morphism $f_V: C\to \PP^2$ 
{onto a degree $a$ plane curve.}
The conjugate subspace $\overline{V}$ is well-defined and gives a complex holomorphic map $f_{\overline{V}}: C\to \PP^2$. 
Notice that, by construction, we have 
\begin{equation}\label{fv}
f_V(p)=\overline{f_{\overline{V}}(\overline p)}.
\end{equation}
Now, the map, 
$$f=(f_V,f_{\overline{V}}): C\to \PP^2\times \PP^2$$ 
induces a morphism birational onto its image, and the curve $T:= f(C)\subset \PP^2\times \PP^2 $ satisfies $j(T)=T$. 
Moreover, thanks to Formula~\eqref{fv} any real point of $C$ gives a point of $T$ corresponding to a twistor fiber.
Since $C$ is a smooth curve defined over $\RR$ with at least a real point, then $C_\RR$ contains infinitely many real points by \cite{gh} and therefore, we conclude that $T$ parametrizes infinitely many twistor fibers.

Now we want to show that the surface 
$$S=\bigcup_{(q,m)\in T} L_{q,m}$$
has bidegree $(a,a)$, where $a= \deg(f_V(C))= \deg(L)$. 
First of all, thanks to~\cite[Corollary 8.8]{ABBS}, it is clear that, if $S\in|\Oo_{F}(a,b)|$, then $a=b$.

In order to compute the bidegree of $S$ we need to intersect with
$\pi^{-1}_2(n)$ where $n$ is a general point of the second $\PP^2$ and represents a general line in the first $\PP^2$.
Now
$$S\cap \pi^{-1}_2(n)=\{(p,n): pm=0, pn=0, qn=0, (q,m)\in T\}$$
Let be $q_1,\ldots, q_a$ the $a$ points of intersection between the curve $f_V(C)$ and the line corresponding to $n$.
For any $1\le i\le a$, let $m_i$ be such that $(q_i,m_i)\in T$ and set $p_i$ the intersection of the lines corresponding to $m_i$ and $n$. Hence it is clear that
$S\cap \pi^{-1}_2(n)=\{(p_1,n), \ldots, (p_a,n)\}$. This proves that the bidegree of $S$ is $(a,a)$.
\end{proof}

Thanks to Theorem~\ref{infinite} and to~\cite[Theorem 7.2]{ABBS}, we can state the following corollary.  

\begin{corollary}
For any $a>0$, there is an integral surface $X\subset\FF$ of bidegree $(a,a)$ containing
infinitely many twistor fibers.
\end{corollary}

We conclude with a topological observation on the set of twistor
fibers contained in a surface containing infinitely many twistor fibers.

\begin{remark}
It is well-known that
the real locus $C_\RR$ of a curve $C$ of genus $g$ consists of $n\le g+1$ disjoint circles, see
\cite[Proposition 3.1]{gh}.
Thus
from the construction in the proof of Theorem \ref{infinite}, we know that, for any $a\ge n+1$, there exists 
an integral surface $S\in |\Oo_{\FF}(a,a)|$ containing
infinitely many twistor fibers and such that the family of twistor fibers is parametrized by $n$ disjoint circles,
up to finitely many identifications (all over 
real points of $\mathrm{Sing}(C)$) and up to finitely many isolated points (at most $|\mathrm{Sing}(C)|$).
\end{remark}

\section*{Bibliography}

\addtocontents{toc}{\protect\setcounter{tocdepth}{-1}}


\end{document}